\documentclass[12pt]{amsart}
\usepackage{times}
\usepackage{amssymb, amsmath}
\usepackage{amsthm}

\newcommand{\bA}{\mathbb{A}}

\newcommand{\sB}{\mathcal{B}}
\newcommand{\sH}{\mathcal{H}}

\newcommand{\sHo}{\mathcal{H}_{0}}

\newcommand{\sK}{\mathcal{K}}

\newcommand{\cstar}{C^*}
\newcommand{\ds}{(A, G, \sigma)}
\newcommand{\cp}{A\rtimes_{\sigma} G}

\theoremstyle{plain}
\newtheorem{thm}{Theorem}
\newtheorem{prop}[thm]{Proposition}
\newtheorem{lemma}[thm]{Lemma}
\newtheorem{cor}[thm]{Corollary}

\theoremstyle{remark}

\begin{document}
\title[The dual structure of systems with finite groups] {The dual structure of crossed product $C^*$-algebras with finite groups}
\author{Firuz Kamalov}
\address{Mathematics Department, Canadian University of Dubai, Dubai, UAE}
\email{firuz@cud.ac.ae}
\date{\today}

\begin{abstract}
We study the space of irreducible representations of a crossed product $C^*$-algebra $\cp$, where $G$ is a finite group. We construct a space $\widetilde{\Gamma}$ which consists of pairs of irreducible representations of $A$ and irreducible projective representations of subgroups of $G$. We show that there is a natural action of $G$ on $\widetilde{\Gamma}$ and that the orbit space $G\backslash \widetilde{\Gamma}$ corresponds bijectively to the dual of $\cp$.

\end{abstract}
\maketitle

\section{Introduction}
Let $A$ be a $\cstar$-algebra and let $G$ be a locally compact group acting as automorphisms of $A$ via a homomorphism $\sigma$ into Aut$(A)$. It has been a long standing problem to describe the ideal structure of the crossed product $\cp$. One approach to describing Prim($\cp$) is to construct a set $X$ whose structure can be understood and then realize Prim($\cp$) as the quotient space of $X$. Perhaps the best example of such approach is given by Williams in \cite{W2}, where $A$ and $G$ are assumed to be abelian. In this case, Prim($\cp$) can be realized as the quotient space of $X=\widehat{A}\times \widehat{G}$. In general, the problem of constructing the appropriate space $X$ seems to be very difficult. Even in special cases when $A$ is Type I or $G$ is amenable the problem remains open \cite{W1}.

The purpose of this paper is to describe the dual space $\widehat{\cp}$ of $\cp$, that is, the set of all unitary equivalence classes of irreducible representations of $\cp$, when $G$ is finite. The study of crossed products involving finite groups goes back to Rieffel \cite{R}. More recently, it was shown by Arias and Latremoliere in \cite{Ar} that every irreducible representation of $\cp$ is induced from an irreducible representation of a certain subsystem. In Section 2, we construct a space $\widetilde{\Gamma}$ which consists of pairs of unitary equivalence classes of irreducible representations of $A$ and irreducible projective representations of certain subgroups of $G$. There is a natural action of $G$ on $\widetilde{\Gamma}$. We define a map $\Phi$ from $\widetilde{\Gamma}$ into the set of equivalence classes of irreducible covariant representations of the dynamical system $\ds$. In Section 3, we show that the map $\Phi$ is surjective. This result is also proved in \cite[Theorem 3.4]{Ar} but we provide an alternative approach. Our main result is Theorem~\ref{thm1}, where we identify $\widehat{\cp}$ with the set of orbits in $\widetilde{\Gamma}$.

Recall that a covariant representation of $\ds$ on a Hilbert space $\sH$ is a pair $(\pi, U)$, where $\pi$ is a non-degenerate representation of $A$ on $\sH$ and $U$ is a homomorphism of $G$ into the unitary group of $\sB
(\sH)$ such that
\[U(s)\pi(a)U(s)^*=\pi(\sigma_{s}a)\]
for all $a\in A$ and $s\in G$. There exists a one to one correspondence between the covariant representations of the system $\ds$ and the nondegenerate representations of $\cp$. Therefore, the study of
representations of $\cp$ is equivalent to that of
covariant representations of $\ds$.

\section{The Action of $G$ on $\Gamma$}
Let $\ds$ be a dynamical system, where $G$ is a finite group. The action of $G$ on $A$ induces a natural action of $G$ on $\widehat{A}$ given by $[\pi] \mapsto [\pi\circ \sigma_s]$ for all $[\pi]\in \widehat{A}$ and $s\in G$. Define $G_{\pi} = \{ s\in G : [\pi]=[\pi\circ\sigma_s]\}$ to be the stability group for each $[\pi]\in \widehat{A}$. Then for each $s\in G_{\pi}$ there is a unitary $V_s$ such that $V_s \pi V_s^* = \pi\circ \sigma_s$. A routine calculation shows that the map $s \mapsto V_s$ defines a projective representation of $G_{\pi}$. Let $\omega$ be the multiplier of the projective representation $V$. The multiplier $\omega$ and the projective representation $V$ do not depend on the choice of $\pi$ but only on the equivalence class $[\pi]$. Let $W_\omega$ be an $\omega$-representation of $G_{\pi}$, then according to \cite{M}, $\overline{W_\omega}$, the adjoint of $W_\omega$, is an $\omega^{-1}$-representation. We can construct a covariant representation of $(A, G_{\pi}, \sigma)$ by
\begin{equation}\label{eq:fingr}
\pi_{\omega}=\pi\otimes 1  \mbox{ and } U_{\omega}=V\otimes \overline{W_\omega}.
\end{equation}
The map $W_\omega \mapsto (\pi_{\omega},U_{\omega})$ sets up a one-to-one correspondence between the set of $\omega$-representations of $G_{\pi}$ and the set of all covariant representations of $(A, G_{\pi}, \sigma)$ of the form $(\pi\otimes 1, V\otimes \overline{W_\omega})$. Moreover, the commutant of $(\pi_{\omega},U_{\omega})$ is isomorphic to the commutant of $W_\omega$ under the canonical correspondence \cite[Lemma 5.2]{T}. In particular, if $W_\omega$ is irreducible, then so is $(\pi_{\omega},U_{\omega})$.

Let $\Gamma$ be the set of all pairs $(\pi, W_{\omega})$, where $\pi$ is an irreducible representation of $A$ and $W_{\omega}$ is an irreducible $\omega$-representation of $G_{\pi}$. There exists a natural action of $G$ on the set $\Gamma$ which we now describe. For each $s\in G$, we have  $G_{\pi\circ\sigma_s}=s^{-1}G_{\pi}s$. So given a projective representation $W_{\omega}$ of $G_{\pi}$ we can construct a projective representation of $G_{\pi\circ\sigma_s}$ by $(s\cdot W_{\omega})(s^{-1}ts)=W_{\omega}(t)$ for all $t\in G_{\pi}$. Thus we can define the action of $G$ on $\Gamma$ by
\[(\pi, W_{\omega})\mapsto (\pi\circ\sigma_s, s \cdot W_{\omega}).\]

In order to establish a connection between $\Gamma$ and $\widehat{\cp}$  we need to extend a representation of $(A, G_{\pi}, \sigma)$ to a representation of $\ds$. We will use Mackey-Takesaki construction of induced representations for this purpose. Since we are working with a finite group $G$ induced representations are easy to describe. Let $H$ be a subgroup of $G$ and let $(\pi, U)$ be a covariant representation of $(A, H, \sigma)$ on a Hilbert space $\sH_0$. Let $\sH$ be the space of all $\sHo$ valued functions $\xi$ on $G$ satisfying $\xi(ts)=U(t)\xi(s)$ for all $t\in H$ and all $s\in G.$ Define $\overline{U}$ to be the homomorphism of $G$ into the unitary group of $\sB (\sH)$ given by
\[(\overline{U}(t)\xi)(s)=\xi(st)\] for all
$\xi\in\sH$ and $s, t\in G$. For each $a\in A$, define an operator
$\overline{\pi}(a)$ on $\sH$ by
\[(\overline{\pi}(a)\xi)(s)=\pi(\sigma_{s}a)\xi(s)\]
for all $\xi\in\sH$ and $s\in G$. Then $(\overline{\pi}, \overline{U})$ is the induced covariant representation of $\ds$.

Let $H$ be a subgroup of $G$ and let $(\pi, U)$ be a representation of  $(A, H, \sigma)$. Let $s\in G$. Define a representation $(\pi\circ\sigma_s, U_s)$ of $(A, s^{-1}Hs, \sigma)$ by $U_s(s^{-1}ts)=U(t)$ for all $t\in H$. We want to establish that $(\pi, U)$ and $(\pi\circ\sigma_s, U_s)$ induce to equivalent representations.


\begin{lemma}\label{Lem0}
 Let $\ds$ be a dynamical system, where $G$ is a finite group. Let $H$ be a subgroup of $G$ and $s\in G$. Suppose that $(\pi, U)$ and $(\pi\circ\sigma_s, U_s)$ are as above and that $(\overline{\pi}, \overline{U})$ and $(\overline{\pi\circ\sigma_s}, \overline{U_s})$ are the corresponding induced representations of $\ds$. Then $(\overline{\pi}, \overline{U})$ is unitarily equivalent to $(\overline{\pi\circ\sigma_s}, \overline{U_s})$.
\end{lemma}
\begin{proof}
Let $\sH$ denote the representation space for $(\overline{\pi}, \overline{U})$ and $\sH_s$ denote the representation space for $(\overline{\pi\circ\sigma_s}, \overline{U_s})$. Define a unitary $V$ from $\sH$ to $\sH_s$ by $(V\xi)(r)=\xi(sr)$ for all $\xi\in\sH$ and $r\in G$. For each $\eta\in \sH_s$,
\begin{align*}
(V\overline{\pi}(a)V^*\eta)(r)\\
&= (\overline{\pi}(a)V^*\eta)(sr)\\
&= \pi(\sigma_{sr}a)(V^*\eta)(sr)\\
&= \pi(\sigma_{sr}a)\eta(r)
= (\overline{\pi\circ \sigma_{s}}(a)\eta)(r)
\end{align*} for all $r\in G$ and $a\in A$. Similarly,
 \[(V\overline{U}(t)V^*\eta)(r)=\eta(rt)=(\overline{U_s}(t)\eta)(r)\]
 for all $t, r\in G$.  It follows that $(\overline{\pi}, \overline{U})$ is equivalent to $(\overline{\pi\circ\sigma_s}, \overline{U_s})$ via the unitary $V$.
\end{proof}

Let $(\pi_\omega, U_{\omega})$ be a representation of $(A, G_\pi, \sigma)$ as in Equation~\ref{eq:fingr}. For each representation of the form $(\pi_{\omega},U_{\omega})$, we can induce to a representation $(\overline{\pi_{\omega}},\overline{U_{\omega}})$ of $\ds$. The commutant of $(\pi_{\omega},U_{\omega})$ is isomorphic to the commutant of $(\overline{\pi_{\omega}},\overline{U_{\omega}})$. In particular, if $(\pi_{\omega},U_{\omega})$ is irreducible, then so is $(\overline{\pi_{\omega}},\overline{U_{\omega}})$.
Let $(\pi_1, W_{\omega_1})$ and $(\pi_2, W_{\omega_2})\in \Gamma$. We will say that $(\pi_1, W_{\omega_1})$ is equivalent to ($\pi_2, W_{\omega_2})$ if $\pi_1$ is unitarily equivalent to $\pi_2$ and $W_{\omega_1}$ is unitarily equivalent to $W_{\omega_2}$. Let $\widetilde{\Gamma}$ be the set of all equivalence classes in $\Gamma$. Note that the action of $G$ on $\Gamma$ induces the action of $G$ on $\widetilde{\Gamma}$.

\begin{lemma}\label{Lem1}
Let $\ds$ be a dynamical system, where $G$ is a finite group. Let $(\pi_1, W_{\omega_1}), (\pi_2, W_{\omega_2})\in \Gamma$ and let $(\overline{\pi_{\omega_1}},\overline{U_{\omega_1}}),(\overline{\pi_{\omega_2}},\overline{U_{\omega_2}})$ be the corresponding representations of $\ds$. If $(\overline{\pi_{\omega_1}},\overline{U_{\omega_1}})$ is unitarily equivalent to $(\overline{\pi_{\omega_2}},\overline{U_{\omega_2}})$,
then
$(\pi_1, W_{\omega_1})$ is equivalent to $(\pi_2\circ\sigma_s, s\cdot W_{\omega_2})$ for some $s\in G$.
\end{lemma}
\begin{proof}
Let $\sH$ and $\sK$ be representation spaces for $(\overline{\pi_{\omega_1}},\overline{U_{\omega_1}})$ and $(\overline{\pi_{\omega_2}},\overline{U_{\omega_2}})$ respectively. Let $\{r_i\}$ be the set of right coset representatives of $G_{\pi_1}$ in $G$. Define $\sH_i=\{\xi\in\sH : \xi(t)=0 \mbox{ for all } t\not\in G_{\pi_1}r_i\}$, i.e. $\sH_i$ is the set of functions in $\sH$ supported on the coset $G_{\pi_1} r_i$. Then $\overline{\pi_{\omega_1}}_{|\sH_i}$ is equivalent to $\pi_{\omega_1}\circ\sigma_{r_i}$ for each $r_i$ and $\overline{\pi_{\omega_1}}$ decomposes as a direct sum of disjoint representations
\[\overline{\pi_{\omega_1}}=\underset{i}{\oplus}\pi_{\omega_1}\circ\sigma_{r_i}\]
Similarly, $\overline{\pi_{\omega_2}}=\underset{j}{\oplus}\pi_{\omega_2}\circ\sigma_{s_j}$, where $\{s_j\}$ is the set of right coset representatives of $G_{\pi_2}$ in $G$. Since $(\overline{\pi_{\omega_1}},\overline{U_{\omega_1}})$ is unitarily equivalent to $(\overline{\pi_{\omega_2}},\overline{U_{\omega_2}})$ there is a unitary $V$ such that $V\overline{\pi_{\omega_1}}=\overline{\pi_{\omega_2}}V$ and $V\overline{U_{\omega_1}}=\overline{U_{\omega_2}}V$. We can view $V$ as a matrix operator with respect to decomposition $\sH=\underset{i}{\oplus}\sH_i$ and $\sK=\underset{j}{\oplus}\sK_j$. Since $\{\pi_1\circ\sigma_{r_i}\}_i$ are mutually inequivalent representations and $\{\pi_2\circ\sigma_{s_j}\}_j$ are also mutually inequivalent, then $V$ is a permutation matrix whose nonzero entries are unitaries. Therefore, there exists a unitary $V_{j1}$ such that $V_{j1}\pi_{\omega_1}=(\pi_{\omega_2}\circ \sigma_{s_j}) V_{j1}$ for some $s_j$. It follows that $\pi_1$ is equivalent to $\pi_2 \circ \sigma_{s_j}$ and $G_{\pi_1}=s_j^{-1}G_{\pi_2}s_j.$ Observe that the restriction of $\overline{U_{\omega_1}}_{|\sH_1}$ to $G_{\pi_1}$ is equivalent to the representation $U_{\omega_1}$ and the restriction of $\overline{U_{\omega_2}}_{|\sK_1}$ to $G_{\pi_2}$ is equivalent to the representation $U_{\omega_2}$. Since $V\overline{U_{\omega_1}}=\overline{U_{\omega_2}}V$, then $V_{j1}\overline{U_{\omega_1}}_{|\sH_1}(r)=\overline{U_{\omega_2}}_{|\sK_j}(r) V_{j1}$ for all $r\in G_{\pi_1}$.  Also $\overline{U_{\omega_2}}_{|\sK_j}(s_j^{-1} t s_j)$ is equivalent to $\overline{U_{\omega_2}}_{|\sK_1}(t)$ for all $t\in G_{\pi_2}$. Therefore, $U_{\omega_1}(s_j^{-1} t s_j)$ is equivalent to $U_{\omega_2}(t)$ for all $t\in G_{\pi_2}$. It follows that $(\pi_1, W_{\omega_1})$ is equivalent to $(\pi_2\circ\sigma_{s_j}, s_j \cdot W_{\omega_2})$.

\end{proof}
Define a map $\Phi$ from $\widetilde{\Gamma}$ into the set of equivalence classes of irreducible covariant representations of $\ds$ by

\begin{equation}\label{eq:Pheta}
\Phi(\pi, W_{\omega})= (\overline{\pi_{\omega}},\overline{U_{\omega}}).
\end{equation}

If $(\pi_1, W_{\omega_1})$ is equivalent to $(\pi_2, W_{\omega_2})$, then  $\Phi(\pi_1, W_{\omega_1})$ is equivalent to $\Phi(\pi_2, W_{\omega_2})$. So $\Phi$ is well defined. The next result follows directly from Lemmas 1 and 2.

\begin{cor}\label{Cor1}
Let $\ds$ be a dynamical system, where $G$ is a finite group. Suppose that $(\pi_1, W_{\omega_1})$ and $(\pi_2, W_{\omega_2})\in \widetilde{\Gamma}$. Then $\Phi(\pi_1, W_{\omega_1})=\Phi(\pi_2, W_{\omega_2})$ if and only if
$(\pi_2, W_{\omega_2})= (\pi_1\circ\sigma_s, s \cdot W_{\omega_1})$ for some $s\in G$.
\end{cor}

\section{The Main Result}

The remaining step in obtaining our main result is to show that the map $\Phi$, as defined in Equation~\ref{eq:Pheta}, is surjective. We first need the following elementary lemma about projections.

\begin{lemma}\label{simpleLem}
Let $\sH$ be a Hilbert space and $\bA$ be a von Neumann algebra in $\sB(\sH)$. Let $p_1$ and $p_2$ be a pair of projections in $\bA$. Suppose that $q=p_1-(p_1\wedge p_2)$. Then $q\wedge p_2= 0$. Moreover, if $p_2$ is a minimal projection, then $(p_1\vee p_2)-p_1$ is a minimal projection in $\bA$.
\end{lemma}
\begin{proof}
Suppose that $qh_1=p_2 h_2$ for some $h_1, h_2 \in \sH$. Since $q\leq p_1$, then $p_1p_2h_2=p_2h_2$. Hence, $(p_1\wedge p_2)h_2=p_2h_2$. It follows that $(p_1\wedge p_2)h_2=qh_1$. But $q\wedge (p_1\wedge p_2)=0$, so $qh_1=0$.

To prove the second part of the statement let $e=(p_1\vee p_2)-p_1$. Suppose there exists a nonzero projection $e'\in \bA$ such that $e'\lneq e$. Then $p_2e'\neq 0$ and $p_2e'\sH\subsetneq p_2\sH$. Let $p_2'$ be the projection onto the closure of the range of $p_2e'$. Then $p_2'\in\bA$ and $p_2'\lneq p_2$ which is a contradiction. It follows that $e$ is a minimal projection.
\end{proof}

Let $(\pi, U)$ be a covariant representation of $\ds$ on a Hilbert space $\sH$. There is a natural action of $G$ on the von Neumann algebra $\pi(A)'$ given by $T\mapsto U(s)TU(s)^*$ for all $T\in \pi(A)'$. We say that the action of $G$ on a von Neumann algebra $\bA$ is ergodic if the only elements of $\bA$ that are fixed by the group action are the scalar multiples of the identity operator. It was shown in \cite[Theorem 3.1]{Ar}, using a powerful result of \cite{L}, that von Neumann algebras which admit ergodic action by a finite group are necessarily finite dimensional. We present this result below with an alternative proof.

\begin{prop}\label{Lem2}
Let $U$ be a unitary representation of a finite group $G$ on a Hilbert space $\sH$. Suppose that $G$ acts ergodically on a von Neumann algebra $\bA$ in $\sB(\sH)$. Then there exists a finite family of minimal projections $p_i\in \bA$ such that $\oplus p_i=1_\sH$.
\end{prop}

\begin{proof}
We will first show that there exists a minimal projection $p\in \bA$ together with a subset $S\subseteq G$ such that $\underset{s_j\in S}{\vee} U(s_j)pU(s_j)^* =1_\sH$ and $(\underset{j\leq i-1}{\vee} U(s_j)pU(s_j)^* )\wedge U(s_i)pU(s_i)^*=0$ for all $s_i\in S$. To this end, let $p\in\bA$ and $S'\subseteq G$ such that $(\underset{j\leq i-1}{\vee} U(s_j)pU(s_j)^* )\wedge U(s_i)pU(s_i)^*=0$ for all $s_i\in S'$. Suppose that $p$ is not a minimal projection. It will be enough to show that there is a projection $p'\in \bA$ and $t\in G-S'$ such that $(\underset{j\leq i-1}{\vee} U(s_j)p'U(s_j)^*) \wedge U(s_i)p'U(s_i)^*=0$ for all $s_i\in S$, where $S=S'\cup \{t\}$. Since $G$ is finite we will eventually obtain a minimal projection.

For each projection $q\in \bA$,  we have $\underset{G}{\sum} U(s)qU(s)^* \in \bA$. Moreover,
\[U(t) (\sum_G U(s)qU(s)^*) U(t)^*=\sum_G U(s)qU(s)^*\]
for all $t\in G$. Since the group action is ergodic, then $\underset{G}{\sum} U(s)qU(s)^*=c 1_\sH$ for some complex number $c$. It follows that
\begin{equation}\label{eq:T}
\underset{G}{\vee} U(s)qU(s)^*= 1_\sH
\end{equation}
for all non zero projections $q\in \bA$. Assume, without loss of generality, that $1_G\in S'$. Moreover, by replacing $p$ with a proper, nonzero subprojection we can assume that $\underset{s\in S'}{\vee} U(s)pU(s)^*<1_\sH$. By Equation~\ref{eq:T},
there is $t \in G$ such that $U(t)pU(t)^* \not\leq \underset{s\in S'}{\vee} U(s)pU(s)^*$. Note that $t\not\in S'$. Let $q=U(t)pU(t)^*-[U(t)pU(t)^*\wedge (\underset{s\in S'}{\vee} U(s)pU(s)^*)]$. By Lemma~\ref{simpleLem}, $q\wedge (\underset{s\in S'}{\vee} U(s)pU(s)^*)=0$. Then $p'=U(t)^*qU(t)$ is the desired projection.

We will now describe how to transform the set of minimal projections $\{U(s_i)pU(s_i)^*\}_{s_i\in S}$ obtained above into a set of orthogonal minimal projections. Let $q_i=U(s_i)pU(s_i)^*$ for all $s_i\in S$. For each $i\geq 2$, define

\[p_i=\underset{1\leq j \leq i}{\vee} q_j - \underset{1\leq j \leq i-1}{\vee} q_j\]
and $p_1=q_1$. Then $p_i \in \bA$ for all $i$, and $p_i \perp p_j$ for all $i\neq j$. Moreover, by the second part of Lemma~\ref{simpleLem}, each $p_i$ is a minimal projection.
\end{proof}

Suppose $(\pi, U)$ is an irreducible representation of $\ds$. Then the action of $G$ on $\pi(A)'$ is ergodic. Applying Proposition~\ref{Lem2} to the algebra $\pi(A)'$ we get that $\pi$ must decompose as a direct sum of finitely many irreducible representations. Let $\rho$ be an irreducible subrepresentation of $\pi$. It follows from \cite[Theorem 3.4]{Ar} that there exists an irreducible $\omega$-representation of $G_{\rho}$ such that $(\pi, U)$ is unitarily equivalent to $(\overline{\rho_{\omega}},\overline{U_{\omega}})$. It follows that the map $\Phi$, as defined in Equation~\ref{eq:Pheta}, is surjective. We are now in position to state our main theorem.

\begin{thm}\label{thm1}
Suppose that $\cp$ is a crossed product $C^*$-algebra, where $G$ is a finite group. Let $\widetilde{\Gamma} \backslash G$ be the set of orbits in $\widetilde{\Gamma}$ under the group action. Then there exists a bijective correspondence between $\widetilde{\Gamma} \backslash G$ and the dual space $\widehat{\cp}$.
\end{thm}
\begin{proof}
Recall that there is a canonical correspondence between the irreducible representations of $\cp$ and $\ds$. By the preceding discussion the map $\Phi: \widetilde{\Gamma} \mapsto \widehat{\cp}$ is surjective. Moreover, by Corollary~\ref{Cor1}, $\Phi(\pi_1, W_{\omega_1})=\Phi(\pi_2, W_{\omega_2})$ if and only if $(\pi_2, W_{\omega_2})$ is in the orbit of $(\pi_1, W_{\omega_1})$.
\end{proof}

\end{document}